\documentclass[12pt]{amsart}

\usepackage[utf8]{inputenc}
\usepackage[T1]{fontenc}
\usepackage[english]{babel}
\usepackage{amsmath,amssymb,amsthm}
\usepackage[margin=1.2in]{geometry}
\usepackage{hyperref}
\usepackage{enumitem}
\usepackage{xcolor}
\usepackage{tikz}
\usetikzlibrary{decorations.pathreplacing,patterns}

\definecolor{ammblue}{HTML}{0A4EB0}
\definecolor{indigo}{HTML}{1B2A6B}

\hypersetup{colorlinks,linkcolor=indigo,citecolor=indigo,urlcolor=ammblue,
  pdftitle={A Random Approach to the Multibonacci Sequence},
  pdfauthor={Belbachir, Zeggada}}

\theoremstyle{plain}
\newtheorem{theorem}{Theorem}[section]
\newtheorem{corollary}[theorem]{Corollary}
\theoremstyle{definition}
\newtheorem{remark}[theorem]{Remark}

\begin{document}

\title{A Random Approach to the Multibonacci Sequence}

\author{Hac\`ene Belbachir}
\address{USTHB, Faculty of Mathematics, RECITS Laboratory,
  Po.\ Box~32, El~Alia, 16111 Bab~Ezzouar, Algiers, Algeria}
\email{hbelbachir@usthb.dz}

\author{Hamza Zeggada}
\address{USTHB, Faculty of Mathematics, RECITS Laboratory,
  Po.\ Box~32, El~Alia, 16111 Bab~Ezzouar, Algiers, Algeria}
\email{hzeggada@usthb.dz}



\date{\today}

\begin{abstract}
This paper presents a random approach to the multibonacci sequence.
We generalise the model introduced by Benjamin, Levin, Mahlburg, and
Quinn, which is based on a random tiling method using dominoes and
squares that leads to the Fibonacci sequence, and which was extended
to the tribonacci case in a previous work by the authors.
Our approach employs tiling with linear $k$-ominoes, $k=1,\ldots,s$,
combined with specific colouring, to generate a weighted multibonacci
sequence.
For a natural random variable~$X$ defined by this model, we establish
the distribution of $X$ in terms of multibonacci numbers and compute
$\mathbb{E}[X] = 2^{s+1}-3$.
\end{abstract}

\maketitle

\section{Introduction}

In their original work~\cite{Benjamin2000}, Benjamin, Levin, Mahlburg,
and Quinn established a probabilistic model based on random tilings of
an infinite board, which naturally generates Fibonacci identities.
In~\cite{Zeggada2024}, the present authors extended this model to
produce tribonacci identities.
The aim of the current paper is to generalise both results to the
multibonacci (or $s$-bonacci) sequence, for any integer $s \geq 2$.

The model introduced in this study generates random tilings of length
$n$ depending only on $n$.

To start, we provide a combinatorial interpretation of multibonacci
numbers. Let $c_n$ denote the number of compositions of $n$ using
parts in $\{1,2,\ldots,s\}$.
We observe that $c_n = U_{n+1}$, where $U_n$ denotes the $n$-th
multibonacci number.
Specifically, $c_1 = 1 = U_2$, $c_2 = 2 = U_3$
(compositions $1{+}1$ and $2$), $c_3 = 4 = U_4$
(compositions $1{+}1{+}1$, $1{+}2$, $2{+}1$, $3$), and so on.
Here $c_s$ is the sum of all preceding terms from $c_1 = U_2$ to
$c_{s-1} = U_s$, and for $n \geq s$ one has
$c_n = c_{n-1} + c_{n-2} + \cdots + c_{n-s}$.
Hence, for $n \geq 1$, the multibonacci number $U_n$ can be defined
combinatorially as the number of ways to tile a board of length $n-1$
using $1$-ominoes, $2$-ominoes, $\ldots$, up to $s$-ominoes.

\begin{remark}
The values $c_1=1=U_2$ and $c_2=2=U_3$ are valid for all $s \geq 2$.
Note that $c_3 = 4 = U_4$ holds for $s \geq 3$; for $s=2$ (Fibonacci)
one has $c_3 = 3 = U_4$ since only parts $1$ and $2$ are allowed,
giving compositions $1{+}1{+}1$, $1{+}2$, $2{+}1$.
\end{remark}

An explicit formula is given in~\cite{Belbachir2006}: $U_0 = 0$ and,
for $n \geq 0$,
\begin{equation}\label{eq:explicit}
  U_{n+1}
  = \sum_{\substack{i_1+2i_2+\cdots+si_s = n}}
    \binom{i_1+i_2+\cdots+i_s}{i_1,\,i_2,\,\ldots,\,i_s}.
\end{equation}

\section{Main results}

The model involves an infinite board containing cells numbered
$1, 2, \ldots$\,. Each cell is independently coloured either black
or white, with probability $1/2$ for each colour. For any colouring
of the first $n$ cells, the probability of that specific colouring
is $(1/2)^n$.

An infinite tiling is defined as coloured strings of black and white
cells of varying lengths. An example of such a board is given in
Figure~\ref{fig:board}.

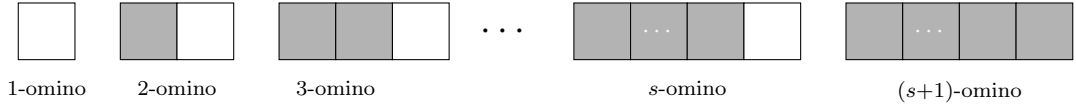
\begin{figure}[h]
\centering
\begin{tikzpicture}[scale=0.75]
  \fill[white](0,0)rectangle(1,1);\draw(0,0)rectangle(1,1);
  \node[below=4pt] at (0.5,0){\tiny $1$-omino};
  \fill[gray!60](1.8,0)rectangle(2.8,1);\draw(1.8,0)rectangle(2.8,1);
  \fill[white](2.8,0)rectangle(3.8,1);\draw(2.8,0)rectangle(3.8,1);
  \node[below=4pt] at (2.8,0){\tiny $2$-omino};
  \fill[gray!60](4.6,0)rectangle(5.6,1);\draw(4.6,0)rectangle(5.6,1);
  \fill[gray!60](5.6,0)rectangle(6.6,1);\draw(5.6,0)rectangle(6.6,1);
  \fill[white](6.6,0)rectangle(7.6,1);\draw(6.6,0)rectangle(7.6,1);
  \node[below=4pt] at (5.6,0){\tiny $3$-omino};
  \node at (8.6,0.5){\large$\cdots$};
  \fill[gray!60](9.8,0)rectangle(10.8,1);\draw(9.8,0)rectangle(10.8,1);
  \fill[gray!60](10.8,0)rectangle(11.8,1);\draw(10.8,0)rectangle(11.8,1);
  \node[white,font=\scriptsize] at (11.3,0.5){$\cdots$};
  \fill[gray!60](11.8,0)rectangle(12.8,1);\draw(11.8,0)rectangle(12.8,1);
  \fill[white](12.8,0)rectangle(13.8,1);\draw(12.8,0)rectangle(13.8,1);
  \node[below=4pt] at (11.8,0){\tiny $s$-omino};
  \fill[gray!60](14.6,0)rectangle(15.6,1);\draw(14.6,0)rectangle(15.6,1);
  \fill[gray!60](15.6,0)rectangle(16.6,1);\draw(15.6,0)rectangle(16.6,1);
  \node[white,font=\scriptsize] at (16.1,0.5){$\cdots$};
  \fill[gray!60](16.6,0)rectangle(17.6,1);\draw(16.6,0)rectangle(17.6,1);
  \fill[gray!60](17.6,0)rectangle(18.6,1);\draw(17.6,0)rectangle(18.6,1);
  \node[below=4pt] at (16.6,0){\tiny $(s{+}1)$-omino};
\end{tikzpicture}
\caption{Available tiles for the tiling of cells $1$ to $n-1$:
a $j$-omino ($j=1,\ldots,s$) consists of $j-1$ black cells followed
by one white cell; the $(s+1)$-omino is entirely black.
}
\label{fig:tiles}
\end{figure}

Let $X$ be the random variable denoting the location of the end of
the first black string of length of the form $sk+(s-1)$, where $k$
is a non-negative integer. In Figure~\ref{fig:board}, with $k=2$,
the first such string has length $3s-1$ and $X = 3s-1$.

\begin{theorem}\label{thm:distribution}
For $n \geq s-1$, the probability that $X = n$ is
\begin{equation}\label{eq:proba}
  P(X=n) = \frac{U_{n-s+2}}{2^{n+1}}.
\end{equation}
As a consequence,
\begin{equation}\label{eq:identity}
  \sum_{n \geq s-1} \frac{U_{n-s+2}}{2^n} = 2.
\end{equation}
\end{theorem}

\begin{proof}
A tiling is categorised as having $X = n$ if and only if the
following conditions are simultaneously satisfied:
\begin{enumerate}[label=(\roman*),leftmargin=2em]
  \item cells $n, n+1, \ldots, n+s-2$ are all black
        (probability $(1/2)^{s-1}$);
  \item cell $n+s-1$ is white (probability $1/2$);
  \item cells $1, 2, \ldots, n-1$ are tiled using the tiles of
        Figure~\ref{fig:tiles}: $j$-ominoes for $j = 1, \ldots, s$
        ($j-1$ black cells followed by one white cell) and the
        $(s+1)$-omino (entirely black).
\end{enumerate}
The number of valid tilings of a strip of length $n-1$ under these
rules equals $U_n$, since the count satisfies the same recurrence
as $(U_n)$ with matching base cases. Each specific tiling of cells
$1$ to $n-1$ has probability $(1/2)^{n-1}$. Combining the three
conditions:
\[
  P(X = n+s-2)
  = U_n \cdot \frac{1}{2^{n-1}} \cdot \frac{1}{2^{s-1}} \cdot \frac{1}{2}
  = \frac{U_n}{2^{n+s-1}}.
\]
Setting $m = n+s-2$ (so $n = m-s+2$) and renaming $m$ as $n$ gives
formula~\eqref{eq:proba}. Identity~\eqref{eq:identity} then follows
from $\sum_{n \geq s-1} P(X=n) = 1$.
\end{proof}

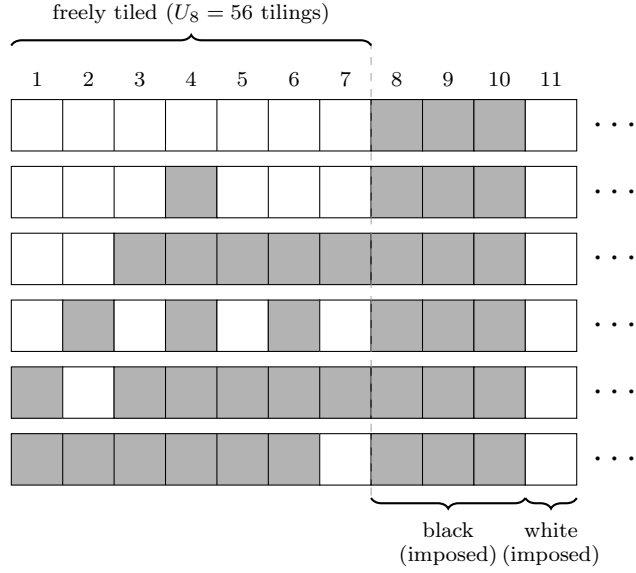
\begin{figure}[h]
\centering
\begin{tikzpicture}[scale=0.68]
  \fill[white](0,6.5)rectangle(7,7.5);\draw(0,6.5)rectangle(7,7.5);
  \foreach \x in {1,2,3,4,5,6}{\draw(\x,6.5)--(\x,7.5);}
  \fill[white](0,5.2)rectangle(3,6.2);\draw(0,5.2)rectangle(3,6.2);
  \foreach \x in {1,2}{\draw(\x,5.2)--(\x,6.2);}
  \fill[gray!60](3,5.2)rectangle(4,6.2);\draw(3,5.2)rectangle(4,6.2);
  \fill[white](4,5.2)rectangle(7,6.2);\draw(4,5.2)rectangle(7,6.2);
  \foreach \x in {5,6}{\draw(\x,5.2)--(\x,6.2);}
  \fill[white](0,3.9)rectangle(2,4.9);\draw(0,3.9)rectangle(2,4.9);
  \draw(1,3.9)--(1,4.9);
  \fill[gray!60](2,3.9)rectangle(7,4.9);\draw(2,3.9)rectangle(7,4.9);
  \foreach \x in {3,4,5,6}{\draw(\x,3.9)--(\x,4.9);}
  \fill[white](0,2.6)rectangle(1,3.6);\draw(0,2.6)rectangle(1,3.6);
  \fill[gray!60](1,2.6)rectangle(2,3.6);\draw(1,2.6)rectangle(2,3.6);
  \fill[white](2,2.6)rectangle(3,3.6);\draw(2,2.6)rectangle(3,3.6);
  \fill[gray!60](3,2.6)rectangle(4,3.6);\draw(3,2.6)rectangle(4,3.6);
  \fill[white](4,2.6)rectangle(5,3.6);\draw(4,2.6)rectangle(5,3.6);
  \fill[gray!60](5,2.6)rectangle(6,3.6);\draw(5,2.6)rectangle(6,3.6);
  \fill[white](6,2.6)rectangle(7,3.6);\draw(6,2.6)rectangle(7,3.6);
  \fill[gray!60](0,1.3)rectangle(1,2.3);\draw(0,1.3)rectangle(1,2.3);
  \fill[white](1,1.3)rectangle(2,2.3);\draw(1,1.3)rectangle(2,2.3);
  \fill[gray!60](2,1.3)rectangle(7,2.3);\draw(2,1.3)rectangle(7,2.3);
  \foreach \x in {3,4,5,6}{\draw(\x,1.3)--(\x,2.3);}
  \fill[gray!60](0,0)rectangle(6,1);\draw(0,0)rectangle(6,1);
  \foreach \x in {1,2,3,4,5}{\draw(\x,0)--(\x,1);}
  \fill[white](6,0)rectangle(7,1);\draw(6,0)rectangle(7,1);
  \foreach \y in {0, 1.3, 2.6, 3.9, 5.2, 6.5}{
    \fill[gray!60](7,\y)rectangle(8,\y+1);\draw(7,\y)rectangle(8,\y+1);
    \fill[gray!60](8,\y)rectangle(9,\y+1);\draw(8,\y)rectangle(9,\y+1);
    \fill[gray!60](9,\y)rectangle(10,\y+1);\draw(9,\y)rectangle(10,\y+1);
    \fill[white](10,\y)rectangle(11,\y+1);\draw(10,\y)rectangle(11,\y+1);
    \node at (11.8,\y+0.5){\large$\cdots$};
  }
  \draw[dashed,gray!70](7,-0.2)--(7,8.6);
  \foreach \x/\lbl in {
    0.5/1,1.5/2,2.5/3,3.5/4,4.5/5,5.5/6,6.5/7,
    7.5/8,8.5/9,9.5/10,10.5/11}{
    \node[above=1pt,font=\tiny] at (\x,7.5){$\lbl$};
  }
  \draw[decorate,decoration={brace,amplitude=4pt,mirror},thick]
    (7,-0.25)--(10,-0.25);
  \node[below=5pt] at (8.5,-0.25){\tiny\shortstack{black\\(imposed)}};
  \draw[decorate,decoration={brace,amplitude=4pt,mirror},thick]
    (10,-0.25)--(11,-0.25);
  \node[below=5pt] at (10.5,-0.25){\tiny\shortstack{white\\(imposed)}};
  \draw[decorate,decoration={brace,amplitude=4pt},thick]
    (0,8.55)--(7,8.55);
  \node[above=4pt] at (3.5,8.55){\tiny freely tiled ($U_8=56$ tilings)};
\end{tikzpicture}
\caption{The $U_8 = 56$ possible tilings of cells $1$--$7$ for
$s=4$, $k=2$, $X=11$ (6 shown): each tiling uses only
$j$-ominoes for $j=1,\ldots,4$ (see Figure~\ref{fig:tiles}),
combined with cells $8$--$10$ black and cell $11$ white,
giving the event $\{X=11=3s-1\}$.
The pattern $\mathrm{W\,B\,B\,B\,B\,W}$
(white from the preceding tile, $4$ consecutive black cells,
white from the following tile) cannot appear in the free cells.}
\label{fig:board}
\end{figure}

\begin{remark}\label{rem:tiling}
The cells $1, 2, \ldots, n-1$ in condition~(iii) are not coloured
independently cell by cell. They are \emph{tiled} using the
$j$-ominoes for $j=1,\ldots,s$ (each consisting of $j-1$ black cells
followed by one white cell). The number of valid tilings of a strip
of length $n-1$ with these tiles equals $U_n$, the $n$-th
multibonacci number, by the recurrence satisfied by $(U_n)$.
In particular, for $s=4$ and $n-1=7$, there are exactly
$U_8 = 56$ tilings, as illustrated in Figure~\ref{fig:board}.
\end{remark}

\begin{theorem}\label{thm:expectation}
The expected value of $X$ is
\begin{equation}\label{eq:EX}
  \mathbb{E}[X]
  = \sum_{n \geq s-1} n\,\frac{U_{n-s+2}}{2^{n+1}}
  = 2^{s+1} - 3.
\end{equation}
In particular, for $s = 2$ (Fibonacci), $\mathbb{E}[X] = 5$;
for $s = 3$ (tribonacci), $\mathbb{E}[X] = 13$;
for $s = 4$ (quadribonacci), $\mathbb{E}[X] = 29$.
\end{theorem}

\begin{proof}
We express $X$ as the sum of three random variables:
\begin{equation}\label{eq:decomp}
  X = (s-1)B + sM + R.
\end{equation}

\textbf{Variable $B$.}
$B$ is a geometric random variable with success probability $1/2$,
so $\mathbb{E}[B] = 2$. The term $(s-1)B$ gives the position of
the end of the first run of $s-1$ consecutive black cells.

\textbf{Variable $M$.}
After the first run of $s-1$ black cells, $M$ counts the number of
complete groups of $s$ consecutive black cells that immediately
follow. The random variable $M+1$ follows a geometric distribution
with parameter $1 - 1/2^s$, giving
$\mathbb{E}[M] = 2^s/(2^s-1) - 1 = 1/(2^s-1)$
and $s\,\mathbb{E}[M] = s/(2^s-1)$.

\textbf{Variable $R$.}
After the first run and the $M$ groups, the next $s$ cells
(positions $(s-1)B+sM+1$ to $(s-1)B+sM+s$) are coloured
independently, giving $2^s$ equally likely outcomes:
\begin{itemize}[leftmargin=2em,itemsep=2pt]
  \item With probability $1/2$: cell $(s-1)B+sM+1$ is white,
        $R = 0$, and a string of length $sk+(s-1)$ ends here.
  \item With probability $1/2^j$ ($j = 1, \ldots, s-1$): the
        first $j$ of these cells are black and the $(j+1)$-th is
        white. The process then restarts, contributing
        $j + \mathbb{E}[X]$ to $R$.
\end{itemize}
By linearity of expectation:
\[
  \mathbb{E}[R]
  = \sum_{j=1}^{s-1} \frac{1}{2^j}\bigl(j + \mathbb{E}[X]\bigr).
\]
Substituting into
$\mathbb{E}[X] = (s-1)\mathbb{E}[B] + s\,\mathbb{E}[M] + \mathbb{E}[R]$:
\begin{equation}\label{eq:linear}
  \mathbb{E}[X]
  = 2(s-1) + \frac{s}{2^s-1}
    + \sum_{j=1}^{s-1}\frac{j}{2^j}
    + \mathbb{E}[X]\sum_{j=1}^{s-1}\frac{1}{2^j}.
\end{equation}
Since $\sum_{j=1}^{s-1}1/2^j = 1 - 1/2^{s-1}$, rearranging gives
\[
  \mathbb{E}[X] \cdot \frac{1}{2^{s-1}}
  = 2(s-1) + \frac{s}{2^s-1} + \sum_{j=1}^{s-1}\frac{j}{2^j}.
\]
Using the finite sum $\sum_{j=1}^{s-1} j/2^j = 2 - (s+1)/2^{s-1}$
and simplifying yields $\mathbb{E}[X] = 2^{s+1}-3$.
\end{proof}

\begin{corollary}[Second moment and variance]\label{cor:variance}
Let
\begin{align*}
  A(t) &= (2t-1)(3t^2-7t+6), \\
  B(t) &= 2(t-1)(4t^3-22t^2+29t-12), \\
  C(t) &= (t-1)^2(2t-1)(2t^2-9t+18).
\end{align*}
Then
\begin{equation}\label{eq:EX2}
  \mathbb{E}[X^2]
  = \frac{s^2\,A(2^s)+s\,B(2^s)+C(2^s)}{2\,(2^s-1)^2},
\end{equation}
and $\mathrm{Var}(X) = \mathbb{E}[X^2] - (2^{s+1}-3)^2$.
\end{corollary}

\begin{corollary}[Fibonacci case, $s=2$]\label{cor:fib}
$\mathbb{E}[X^2] = 853/9$ and $\mathrm{Var}(X) = 628/9$.
\end{corollary}

\begin{corollary}[Tribonacci case, $s=3$]\label{cor:trib}
$\mathbb{E}[X^2] = 54840/49$ and $\mathrm{Var}(X) = 46559/49$.
\end{corollary}

\begin{corollary}[Quadribonacci case, $s=4$]\label{cor:quad}
$\mathbb{E}[X^2] = 2182591/225$ and $\mathrm{Var}(X) = 1993366/225$.
\end{corollary}

\begin{proof}
We apply the decomposition $X=(s-1)B+sM+R$ from
Theorem~\ref{thm:expectation}.
Since $R$ satisfies the restart property ($R=0$ with probability
$1/2$; $R=j+X'$ with probability $1/2^j$ for $j=1,\ldots,s-1$,
where $X'$ is an independent copy of $X$), we get
\[
  \mathbb{E}[R^2]
  = S_3 + 2\,\mathbb{E}[X]\,S_1
    + (\mathbb{E}[X])^2 S_2 + \mathbb{E}[X^2]\,S_2,
\]
where $S_1=2-(s+1)/2^{s-1}$, $S_2=1-1/2^{s-1}$, and
$S_3=6-(s^2+2s-2)/2^{s-1}$.
Setting $W=(s-1)B+sM$ and using
$\mathbb{E}[B^2]=6$, $\mathbb{E}[M^2]=(2^s+1)/(2^s-1)^2$,
and the independence of $B$ and $M$, one computes
$\mathbb{E}[X^2]=\mathbb{E}[W^2]+2\,\mathbb{E}[W]\,\mathbb{E}[R]
+\mathbb{E}[R^2]$.
Solving for $\mathbb{E}[X^2]$ (using $1-S_2=1/2^{s-1}$) yields
formula~\eqref{eq:EX2} of Corollary~\ref{cor:variance}.
\end{proof}

\end{document}